\documentclass[a4paper]{amsart}
\usepackage[latin1]{inputenc}
\usepackage{amssymb}
\usepackage{amsmath}
\usepackage{mathrsfs}
\usepackage{eufrak}
\usepackage{amsthm}
\usepackage{amsfonts}
\usepackage{textcomp}
\usepackage{graphicx}
\usepackage{paralist}
\usepackage[shortlabels]{enumitem}
\usepackage{hyperref}
\usepackage{comment}
\usepackage[arrow, matrix, curve]{xy}
\usepackage{tikz} \usetikzlibrary{matrix,patterns,shapes,decorations.pathmorphing,decorations.pathreplacing,calc,arrows} \usepackage{tikz-cd} 
\usepackage{tkz-graph}

\newtheorem*{corollary*}{Corollary}
\newtheorem{theorem}{Theorem}[section]

\newtheorem{question}[theorem]{Question}

\newtheorem*{claim*}{Claim}
\newtheorem{conjecture}[theorem]{Conjecture}

\theoremstyle{definition}

\newtheorem*{theorem }{Theorem}

\newtheorem{problem}[theorem]{Problem}

\theoremstyle{remark}

\numberwithin{equation}{theorem}

\makeatletter
\renewcommand*\env@matrix[1][\
arraystretch]{%
  \edef\arraystretch{#1}%
  \hskip -\arraycolsep
  \let\@ifnextchar\new@ifnextchar
  \array{*\c@MaxMatrixCols c}}
\makeatother

\renewcommand{\mod}{\operatorname{mod}}

\newcommand{\Ext}{\operatorname{Ext}}
\newcommand{\End}{\operatorname{End}}

\newcommand{\add}{\operatorname{\mathrm{add}}}

\renewcommand{\mod}{\operatorname{mod}}

\newcommand{\domdim}{\operatorname{domdim}}

\newcommand{\idim}{\operatorname{idim}}

\newcommand{\gldim}{\operatorname{gldim}}

\begin{document}

\title{Cluster tilting modules for local algebras}
\date{\today}



\author[R.~Marczinzik]{Ren\'e Marczinzik}%
\address[R.~Marczinzik]{Mathematical Institute of the University of Bonn, Endenicher Allee 60, 53115 Bonn, Germany}
\email{marczire@math.uni-bonn.de}
\author[D. Owens]{Daniel Owens}%
\address[D. Owens]{Mathematical Institute of the University of Bonn, Endenicher Allee 60, 53115 Bonn, Germany}
\email{owensd2@tcd.ie}

\begin{abstract}
We give the first example of a non-trivial cluster tilting module in a local finite dimensional algebra.
To do this, we give an explicit calculation of the corresponding higher Auslander algebra by quiver and relations using the GAP-package QPA. 
We discuss related problems and conjectures for local finite-dimensional algebras.
\end{abstract}

\maketitle
\section*{Introduction}

We always assume that algebras $A$ are finite dimensional over a field $K$ and modules are finitely generated right modules unless otherwise stated.
Recall that an $A$-module $M$ is called \emph{$n$-cluster tilting} for $n \geq 1$ if 
$$\add M= \{ X \in \mod A \mid \Ext_A^i(X,M)=0 \ \forall \ 0 <i<n \}= \{ X \in \mod A \mid \Ext_A^i(M,X)=0 \ \forall \ 0 <i<n \}.$$
The \emph{dominant dimension} $\domdim A$ of an algebra $A$ with minimal injective coresolution 
$$0 \rightarrow A \rightarrow I^0 \rightarrow I^1 \rightarrow \cdots$$ is defined as the minimal $n$ such that $I^n$ is not projective, or as infinite if no such $n$ exists.
$A$ is called a \emph{higher Auslander algebra} if $\gldim A \leq n \leq \domdim A$. 
The $n=1$ case in the definition of cluster tilting modules means that there exists an $A$-module such that every indecomposable $A$-module is a direct summand of $M$, that is precisely when $A$ is representation-finite.
The classical \emph{Auslander correspondence} \cite{Aus} states that there is a bijective correspondence up to Morita equivalence between representation-finite algebras and Auslander algebras, that is algebras $B$ with $\gldim B \leq 2 \leq \domdim B.$
Iyama generalised this correspondence to the so called \emph{higher Auslander correspondence} \cite[Theorem 2.6]{Iya3}:

\begin{theorem} 
(Iyama) \label{higherAuscorr}
Let $n \geq 1$. There is a bijective correspondence between algebras $A$ with an $n$-cluster tilting module $M$
and higher Auslander algebras $B$ with $\gldim B \leq n+1 \leq \domdim B$ up to Morita equivalence.
In particular, a generator-cogenerator $M$ over a non-semisimple ring-indecomposable algebra $A$ is an $n$-cluster tilting module if and only if $B:=\End_A(M)$ is a higher Auslander algebra with $\gldim B= \domdim B=n+1$.
\end{theorem}
Here a module $M$ is called \emph{generator-cogenerator} if all indecomposable projective and all indecomposable injective modules are direct summands of $M$. Note that by definition all $n$-cluster tilting modules are generator-cogenerators.

The study of cluster-tilting modules and subcategories is an important problem with many connections to other areas in algebra. Those concepts were first introduced by Iyama in \cite{Iya1} and \cite{Iya2}. We refer for example to the surveys \cite{Iya3} and \cite{GLS} for an overview on cluster tilting modules and higher Auslander algebras and applications to fields such as Lie theory and cluster algebras.
The study of cluster tilting modules and higher Auslander algebras is now widely known as \emph{higher Auslander-Reiten theory}.
While cluster tilting modules and subcategories were studied for a wide range of algebras, so far no non-trivial example has been found for local algebras.
Here we restrict to finite dimensional local $K$-algebras $A$ that are quiver algebras, or equivalently $A/m \cong K$, where $m$ is the maximal ideal of $A$.
Note that over an algebraically closed field every local algebra is a quiver algebra and thus, over algebraically closed fields, this assumption is not a restriction.
We refer to $1$-cluster tilting modules as the trivial case since their existence corresponds to the classical representation-theoretic problem when a finite dimensional algebra is representation-finite, which is solved, see for example \cite{GR}.
It is well known and easy to see that the only local finite dimensional reprentation-finite quiver algebras are those with exactly one arrow, or equivalently those isomorphic to $K[x]/(x^n).$

In this note we give the first example of a local algebra with a non-trivial cluster tilting module.

\begin{theorem} \label{maintheorem}
Let $K$ be a field of characteristic 0.
Let $A$ be the local algebra $A=KQ/I$ with the following quiver and relations:
\begin{align*}
    Q &:= 
    \begin{tikzcd}
        \bullet
        \arrow["b", from=1-1, to=1-1, loop, in=325, out=35, distance=10mm]
        \arrow["a", from=1-1, to=1-1, loop, in=145, out=215, distance=10mm]
    \end{tikzcd} \\
    I &:= (a^2, ab+b^2+b^2a) 
\end{align*}
and let $M := DA \oplus \tau_2 DA  \oplus \tau_2^2 DA  \oplus \tau_2^3 DA  \oplus \tau_2^4 DA$, where $\tau_2 := \tau \Omega^1$.
Then $M$ is a 2-cluster tilting module.

\end{theorem}

Our proof calculates the endomorphism ring of this module $M$ and shows that the endomorphism algebra has global dimension 3 and dominant dimension 3.
Then $M$ is 2-cluster tilting using the higher Auslander correspondence \ref{higherAuscorr}.

The proof will be given using the GAP-package QPA \cite{QPA}.
We explicitly calculate the endomorphism ring of the module $M$ in the theorem by quiver and relations and use the higher Auslander correspondence due to Iyama to show that it is indeed 2-cluster tilting.
For the calculation of the endomorphism ring we provide a program in the appendix using the GAP-package QPA that improves the existing QPA-algorithm "EndOfModuleAsQuiverAlgebra" in the sense that it is much faster and also works over infinite fields.

\section{Calculation of the endomorphism ring in QPA}

\begin{theorem}
Let $K$ be a field of characteristic 0.
Let $A$ and $M$ be as in Theorem \ref{maintheorem}.
Then $B:=\End_A(M)$ is isomorphic to $KP/J,$ where
$P$=
\[\begin{tikzcd}
	1 && 2 \\
	& 3 \\
	4 && 5
	\arrow["a"{description}, from=1-1, to=1-3]
	\arrow["b"{description}, from=1-1, to=3-1]
	\arrow["c"{description}, from=1-3, to=2-2]
	\arrow["d"{description}, from=1-3, to=3-3]
	\arrow["e"{description}, from=2-2, to=3-1]
	\arrow["f"{description}, shift left=3, from=3-1, to=1-1]
	\arrow["g"{description}, shift right=3, from=3-1, to=1-1]
	\arrow["h"{description}, from=3-1, to=3-3]
	\arrow["i"{description}, shift left=3, from=3-3, to=1-3]
	\arrow["j"{description}, shift right=3, from=3-3, to=1-3]
\end{tikzcd}\]

\begin{equation*}
\begin{split}
    J &=(e f,\quad   i c,\quad 
    - b g- a c e g+  b g b f,\quad 
    j d+  j c e h+  j d j d+2 j d i d+  i d j d, \\
    &- d j d -\frac{1}{2} c e g b h- d j c e h+  d j d j d,\quad 
    b h+  a c e h i d,\quad
   - a d- a d i d+  b f a c e h,\\
   &- d i+  c e h j+  c e g b f a,\quad
    e h j c+  e g b f a c,\quad
   - f b+  h j c e+  g b f a c e, \\
&g a- h j- h i d j+  f a c e h i)
\end{split}
\end{equation*}
Furthermore, $B$ has global and dominant dimension equal to three.
\end{theorem}

\begin{proof}
We show the result first for $K= \mathbb{Q}$, the rationals.
We do this using the GAP-package QPA with the additional code in the appendix of this article for the command ``QuiverAndRelationsOfEndOfModule".
Here are the relevant commands:
\begin{verbatim}
K := Rationals; 
Q := Quiver(1,[[1,1,"a"],[1,1,"b"]]);
KQ := PathAlgebra(K,Q);
AssignGeneratorVariables(KQ);
rel := [a^2, a*b+b^2+b^2*a];
A := KQ/rel;
Dimension(A);
IsSelfinjectiveAlgebra(A);
RegA := DirectSumOfQPAModules(IndecProjectiveModules(A));
CoRegA := DirectSumOfQPAModules(IndecInjectiveModules(A));
U1 := DTr(NthSyzygy(CoRegA,1));
U2 := DTr(NthSyzygy(U1,1));
U3 := DTr(NthSyzygy(U2,1));
U4 := DTr(NthSyzygy(U3,1));
IsProjectiveModule(U4);
M := DirectSumOfQPAModules([CoRegA,U1,U2,U3,U4]);
B := QuiverAndRelationsOfEndOfModule(M,20)[1];
GlobalDimensionOfAlgebra(B,3);
DominantDimensionOfAlgebra(B,3);
QQ:=QuiverOfPathAlgebra(B);
Display(QQ);
rel:=RelatorsOfFpAlgebra(B);
Dimension(B);
\end{verbatim}
Here is some explanation:
We first define the algebra by quiver and relations and after checking that it has vector space dimension 6 and is not selfinjective, we define the regular module $A$ and its dual $DA$ and then the module $M= DA \oplus \tau_2 DA  \oplus \tau_2^2 DA  \oplus \tau_2^3 DA  \oplus \tau_2^4 DA$.
We check that $\tau_2^4 DA \cong A$, so that $M$ is indeed a generator-cogenerator.
Then we calculate the endomorphism ring $B=\End_A(M)$ and check that the global dimension and the dominant dimension are indeed equal to 3.
Then we display the quiver of $B$ and its relations. Note that the output gives 71 relations instead of just 11 as in the theorem. 
We eliminated unnecessary relations by using the fact that one can delete a relation if the dimension of the algebra after deletion does not change.
To see that it is really the same algebra, note that the above calculation shows 
the vector space dimension of $B$ is equal to 165.
The next code shows that the quiver algebra with the same quiver and only the 11 relations as in the theorem (which are a subset of the 71 relations as in the first output) has the same vector space dimension and thus those two algebras must be isomorphic:
\begin{verbatim}
Q:=Quiver( ["v1","v2","v3","v4","v5"], [["v1","v2","a10"],["v1","v4","a9"],
["v2","v3","a8"],["v2","v5","a7"],["v3","v4","a6"],["v4","v1","a4"],
["v4","v1","a5"],["v4","v5","a3"],["v5","v2","a1"],["v5","v2","a2"]] );
KQ:=PathAlgebra(Rationals,Q);AssignGeneratorVariables(KQ);
rel:=
[ (1)*a6*a5, (1)*a2*a8,(-1)*a9*a4+(-1)*a10*a8*a6*a4+(1)*a9*a4*a9*a5,  
  (1)*a1*a7+(1)*a1*a8*a6*a3+(1)*a1*a7*a1*a7+(2)*a1*a7*a2*a7+(1)*a2*a7*a1*a7,
  (-1)*a7*a1*a7+(-1/2)*a8*a6*a4*a9*a3+(-1)*a7*a1*a8*a6*a3+(1)*a7*a1*a7*a1*a7,
  (1)*a9*a3+(1)*a10*a8*a6*a3*a2*a7,
  (-1)*a10*a7+(-1)*a10*a7*a2*a7+(1)*a9*a5*a10*a8*a6*a3,
  (-1)*a7*a2+(1)*a8*a6*a3*a1+(1)*a8*a6*a4*a9*a5*a10,
  (1)*a6*a3*a1*a8+(1)*a6*a4*a9*a5*a10*a8,
  (-1)*a5*a9+(1)*a3*a1*a8*a6+(1)*a4*a9*a5*a10*a8*a6,
  (1)*a4*a10+(-1)*a3*a1+(-1)*a3*a2*a7*a1+(1)*a5*a10*a8*a6*a3*a2];
C:=KQ/rel;Dimension(C);
\end{verbatim}

Note that we changed the names of the arrows in the theorem to make the presentation slightly prettier.
QPA also verified for us that the global and dominant dimension of $B$ are indeed equal to three for $K=\mathbb{Q}.$
What is left to show is that this is still correct over an arbitrary field $K$ of characteristic 0.
To see this let $A_K=KQ/I= K \otimes_{\mathbb{Q}} \mathbb{Q}Q/I$ and $M_K=K \otimes_{\mathbb{Q}} M$.
Then $\End_{A_K}(M_K)\cong K \otimes_{\mathbb{Q}} \End_A(M)$, which shows that the quiver and relations are still fine for field extensions of $\mathbb{Q}$.
Now we use that the global dimension and dominant dimension do not change for quiver algebras under field extensions (see \cite[Lemma 5]{Mue} and \cite[Corollary 18]{ERZ}), which finishes the proof.
\end{proof}
As explained before, Theorem \ref{maintheorem} now follows immediately from the higher Auslander correspondence \ref{higherAuscorr}.

\section{Remarks, Open problems and questions}
The algebra $A$ in Theorem \ref{maintheorem} was first found around 2018 in experiments by Jan Geuenich and Rene Marczinzik in searching local finite dimensional non-selfinjective algebras with $\Ext_A^1(D(A),A)=0$.
Rene Marczinzik then noted that this algebra $A$ has a 2-precluster tilting module  $M$ (in the sense of \cite{IS}) and conjectured that this module is indeed 2-cluster tilting by observing that the Cartan determinant of the endomorphism ring is equal to one. 
The existing QPA command ``EndOfModuleAsQuiverAlgebra"was too slow to calculate the endomorphism ring, even when using a supercomputer and waiting several weeks.
It was then posed as a question on mathoverflow, see \cite{MO}, whether $M$ is indeed a 2-cluster tilting module. \O yvind Solberg found a solution that works over a fixed finite field using relative homological algebra and QPA, but calculating the corresponding higher Auslander algebra remained open. It would be interesting to find a human proof of Theorem \ref{maintheorem} that works for arbitrary fields of any characteristic.

We experimented with many other local algebras and surprisingly never found another example with a non-trivial cluster tilting module.
We end this article with several conjectures, questions and problems that are based on those experiments.

\begin{conjecture} \label{conj}
Let $A$ be a local commutative algebra with an $n$-cluster tilting module. Then $n=1$ and $A \cong K[x]/(x^m)$ for some $m \geq 1$.
\end{conjecture}

Recall that a generator-cogenerator $M$ is called \emph{precluster tilting} 
if $B:=\End_A(M)$ satisfies
$\idim B \leq n \leq \domdim B$ for some $n \geq 2$. 
It is known that a precluster tilting module $M$ is even cluster tilting if and only if $B=\End_A(M)$ has finite global dimension, see \cite{IS}.
The Cartan determinant conjecture predicts that the determinant of the Cartan matrix of an algebra of finite global dimension is always equal to one 1, see \cite{Z}. On the other hand, many experiments suggest that having Cartan determinant equal to 1  characterises when an endomorphism ring of a precluster tilting module has finite global dimension. This motivates our next question:
\begin{question}
Let $A$ be an algebra with precluster tilting module $M$ and $B:=\End_A(M).$
Do we have the following equivalence: $M$ is cluster tilting if and only if the Cartan determinant of $B$ is equal to 1?
\end{question}
A positive answer to this question would make checking whether a given generator-cogenerator is cluster tilting very quick and easy since checking the precluster tilting condition is reduced to calculating $\Ext$ and higher Auslander-Reiten translates (see \cite{IS}), and calculating the Cartan matrix is reduced to calculating Hom spaces between the indecomposable summands of $M$.

We pose three further questions that are related to the famous Tachikawa conjectures restricted to local algebras, see, for example, \cite{Yam} for a survey discussing the Tachikawa conjectures.

\begin{question}
If $A$ is local and not selfinjective, can we have $\Ext_A^i(D(A),A)=0$ for $i=1,2$?    
\end{question}
The answer is negative for algebras with radical cube zero, see \cite{As}.
Note that if we always have $\Ext_A^1(D(A),A) \neq 0$ or $\Ext_A^2(D(A),A) \neq 0$ for a local non-selfinjective algebra $A$, then there can not be any $n$-cluster tilting module for $n \geq 3$ for such algebras.

\begin{question}
     If $A$ is local and commutative and non-selfinjective, can we have $\Ext_A^1(D(A),A)=0$?
\end{question}
Again, the answer is negative for radical cube zero algebras by \cite{As}.
Note that if we always have $\Ext_A^1(D(A),A) \neq 0$ for all such algebras, then this would also prove our conjecture \ref{conj} in the non-selfinjective case.
\begin{question}
    
If $A$ is local and selfinjective, can we have $\Ext_A^1(M,M)=0$ for a non-projective $A$-module $M$?
\end{question}
The question is known to have a negative answer for radical cube zero symmetric algebras by results of Hoshino, see \cite{Ho}. If we indeed always have $\Ext_A^1(M,M) \neq 0$ for non-projective $M$ in a local selfinjective algebra, this would mean that there can not exist any non-trivial cluster tilting modules in such algebras.

We discovered the first non-trivial cluster tilting module for a local algebra in this article. Despite many attempts, we did not find further examples, so we pose this as a final problem:
\begin{problem}
Find local finite dimensional algebras with a $n$-cluster tilting module for $n \geq 2$ of arbitrary large vector space dimensions.
\end{problem}
\newpage 
\section{Appendix: QPA code}
Here we present the QPA-code in order to calculate the endomorphism ring of basic modules over quiver algebras. 

\begin{verbatim}
LoadPackage("qpa");

#######################
# Auxiliary Functions #
#######################

DeclareOperation("VectorBasis",[IsField, IsList]);
InstallMethod( VectorBasis,
    "for a base field and a set of generators",
    [ IsField, IsList], 0,
    function( K , generators )

    local V;
    V := VectorSpace(K, generators, "basis");
    return Basis(V, generators);
end);

DeclareOperation("Relation",[IsList, IsList, IsDictionary, IsPositionalObjectRep, IsField]);
InstallMethod( Relation,
    "for a base field and a set of generators",
    [IsList, IsList, IsDictionary, IsPositionalObjectRep, IsField], 0,
    function(basis, coeffs, path_dict, elt_path, K)

    local relation, i;
    if coeffs = fail then
        return fail;
    else
        relation := elt_path;
        for i in [1..Length(basis)] do
            if coeffs[i] <> Zero(K) then
                relation := relation - (coeffs[i] * LookupDictionary(path_dict, basis[i]));
            fi;
        od;
    fi;
    return relation;
end);

DeclareOperation("MatrixByDiagonalBlocks",[IsList, IsRing]);
InstallMethod( MatrixByDiagonalBlocks,
    "for a list of diagonal blocks of a matrix, and a ring",
    [IsList, IsRing], 0,
    function(blocks, R)

    local mat, n, m, c_index, i, a, b;
    n := Length(blocks);
    c_index := [0];

    for i in [1..n] do
        Add(c_index, Size(blocks[i]) + c_index[i]);
    od;
    m := Last(c_index);
    mat := NullMat(m,m,R);

    for i in [1..n] do
        a := c_index[i]+1;
        b := c_index[i+1];
        mat{[a..b]}{[a..b]} := blocks[i];
    od;
    return mat;
end);

DeclareOperation("IdempotentsOfEndOverAlgebra",[IsPathAlgebraMatModule]);
InstallMethod( IdempotentsOfEndOverAlgebra,
    "for a representations of a quiver",
    [IsPathAlgebraMatModule], 0,
    function(M)

    local K, bsids, idemps, maps, mat, i;
    K := LeftActingDomain(M);
    bsids := BlockSplittingIdempotents(M);

    idemps := [];
    maps := NullMat(Dimension(M),Dimension(M),K); 
    for i in [1..Length(bsids)] do
        mat := MatrixByDiagonalBlocks(bsids[i]!.maps, K);
        Add(idemps, Immutable(mat));
    od;
        
    return idemps;
end);

#################
# Main Function #
#################

# M - module
# max_length - maximum path length for search in quiver
DeclareOperation("QuiverAndRelationsOfEndOfModule",[ IsPathAlgebraMatModule, IsInt]);
InstallMethod( QuiverAndRelationsOfEndOfModule,
    "for a representations of a quiver and a maximum search depth",
    [ IsPathAlgebraMatModule, IsInt], 0,
    function( M , max_length )

    local K, endo, radendo, radendo2, idemps, n_idemps, O, Arrows, n_arrows, 
        quiver_data, reordered_arrows, matrix, Q, KQ, vertex_module_list, 
        path_dict, Relations, Elts, len, new_elts, generators, basis, coeffs
        x, y, i, j, k, a, x_path, y_path, B;

    K := LeftActingDomain(M);

    endo := EndOverAlgebra(M); 
    radendo := RadicalOfAlgebra(endo); 
    radendo2 := ProductSpace(radendo,radendo); 
    idemps := IdempotentsOfEndOverAlgebra(M);
    n_idemps := Size(idemps);
    O := Immutable(Zero(endo));

    ##########
    # Quiver #
    ##########

    Arrows := [];
    for x in Basis(radendo) do
        if not (x in radendo2) then Add(Arrows, Immutable(x)); fi;
    od;

    n_arrows := Size(Arrows);

    # the arrows need to be in the order given below so we can match them to 
    # the path algebra elements later
    quiver_data := [];
    reordered_arrows := [];
    matrix := NullMat(n_idemps,n_idemps);
    for i in [1..n_idemps] do
        for j in [1..n_idemps] do
            for k in [1..n_arrows] do
                if idemps[i] * Arrows[k] * idemps[j] = Arrows[k] then
                    Add(quiver_data, [i,j, Concatenation("a", String(k))]);
                    Add(reordered_arrows, Arrows[k]);
                    matrix[i][j] :=  matrix[i][j] + 1;
                fi;
            od;
        od;
    od;
    Arrows := reordered_arrows;

    Q:=Quiver(n_idemps, quiver_data);

    #############
    # Relations #
    #############

    KQ := PathAlgebra(K, Q);

    vertex_module_list:=[]; # vertex_module_list[i] = module associated to vertex i
    AssignGeneratorVariables(KQ);
    path_dict:=NewDictionary(O, true, endo);
    for i in [1..n_idemps] do
        AddDictionary(path_dict, idemps[i], VerticesOfPathAlgebra(KQ)[i]);
        Add(vertex_module_list, 
          [VerticesOfPathAlgebra(KQ)[i], Image(BlockSplittingIdempotents(M)[i])]);
    od;
    for i in [1..n_arrows] do
        AddDictionary(path_dict, Arrows[i], ElementOfPathAlgebra(KQ, ArrowsOfQuiver(Q)[i]));
    od;

    Relations := [];
    Elts := [Arrows];

    len := 1;
    while Last(Elts) <> [] and len <= max_length do
        new_elts := [];
        for x in Last(Elts) do
            for i in [1..n_arrows] do
            y_path := LookupDictionary(path_dict, x) * LookupDictionary(path_dict, Arrows[i]);
                if not IsZero(y_path) then
                    y := x * Arrows[i];
                    generators := Concatenation(Concatenation(Elts), new_elts);
                    basis := VectorBasis(K, generators);
                    coeffs := Coefficients(basis, y);
                    if  coeffs = fail then
                        Add(new_elts, y);
                        AddDictionary(path_dict, y, y_path);
                    else
                        Add(Relations, Relation(basis, coeffs, path_dict, y_path, K));
                    fi;
                fi;
            od;
        od;
        Add(Elts, new_elts);
        len := len + 1;
    od;

    if len > max_length then
        Display("Warning: search exceeded max_length; relations may not be complete.\n");
    fi;

    B := KQ/Relations;

    return [B, Q, matrix, quiver_data, vertex_module_list, Relations];
end);
\end{verbatim}

\section{Acknowledgment} 
The first author thanks \O yvind Solberg for his constant support regarding questions and suggestions on QPA. He also thanks Jan Geuenich for the fun experiments done back in 2018 on finding exotic finite dimensional algebras. The QPA-code to calculate the endomorphism ring of a module is based on an earlier verison of  \O yvind Solberg that is already available in QPA.

\end{document}